\documentclass[12pt]{amsart}
\usepackage[margin=2.9 cm]{geometry}
\usepackage{booktabs} 

\usepackage{amssymb}
\usepackage{xcolor}
\usepackage{epstopdf}
\usepackage{graphicx}
\usepackage{tikz-cd}
\usepackage{todonotes}
\usepackage[normalem]{ulem}
\usepackage[caption=false]{subfig}
\usepackage[numbers,sort&compress]{natbib}
\bibpunct[, ]{[}{]}{,}{n}{,}{,}

\theoremstyle{plain}
\newtheorem{theorem}{Theorem}[section]
\newtheorem{lemma}[theorem]{Lemma}
\newtheorem{corollary}[theorem]{Corollary}
\newtheorem{proposition}[theorem]{Proposition}

\theoremstyle{definition}

\newtheorem{definition}[theorem]{Definition}
\newtheorem{example}[theorem]{Example}

\theoremstyle{remark}
\newtheorem{remark}{Remark}

\newcommand{\Z}{\mathbb Z}

\usepackage{graphicx}
\usepackage{tabulary}

\begin{document}

	\title{Reversibility and  symmetry of affine toral automorphisms}
	
    \author[K. Banerjee]{Kuntal Banerjee}
   
    \address{Department of Mathematics, Presidency University, 86/1, College Street, Kolkata - 700073, West Bengal, India. 
	}
    \email{kbanerjee.maths@presiuniv.ac.in}

     \author[A. Bhattacharyya]{Anubrato Bhattacharyya} 
    \address{Department of Mathematics, Presidency University, 86/1, College Street, Kolkata - 700073, West Bengal, India. 
	}
    \email{anubrato02@gmail.com}
    
\author[K. Gongopadhyay]{Krishnendu Gongopadhyay} 

\address{Indian Institute of Science Education and Research (IISER) Mohali, Knowledge City,  Sector 81, S.A.S. Nagar 140306, Punjab, India}
\email{krishnendu@iisermohali.ac.in}

    \author[S. Mondal]{Subhamoy Mondal}
    \address{Department of Mathematics, The Heritage college,  West Chowbaga Rd, Anandapur, Kolkata - 700107, West Bengal, India. 
	}
        \email{subhamoy.mondal@thc.edu.in,~pranaymondal.math@gmail.com}
	
\begin{abstract}
We study reversibility and strong reversibility of affine automorphisms of the two-torus $\mathbb T^{2}$, written as
$f_{A,\bar{a}}(\bar{x})=A\bar{x}+\bar{a} \ (\mathrm{mod}\ \mathbb{Z}^2)$.
We derive explicit criteria for the reversibility of such maps in terms of the matrix $A$ and the translation  $\bar{a}$. If $1$ is not an eigenvalue of $A$, reversibility of the affine map coincides with reversibility of $A$.
When $1$ is an eigenvalue, additional arithmetic obstructions appear.
We also provide a simple geometric condition, based on Pick's Theorem, that guarantees the existence of fixed points, along with a description of the dynamics of affine toral automorphisms. We also  compute the entropy and characterize when conjugacy classes in the affine group are finite or uncountable.

\end{abstract}


\date{\today}
	\maketitle

\section{Introduction}\label{intro}

Reversibility phenomena play a central role in the study of dynamical systems and geometric group actions. The concept of reversibility originates from physical problems and classical dynamics, including systems such as the $n$-body problem, dynamics of billiards etc. The concept can be traced back to the volume of Fricke and Kelin \cite{FK}  and the work of Birkhoff on the three-body problem \cite{GB}. These notions arise naturally in the study of symmetries of dynamical systems, particularly time-reversal symmetries, and have been extensively explored in settings such as one-dimensional dynamics, complex-analytic maps, and geometric transformations. Reversibility has been investigated from a variety of perspectives in many classes of groups,  see, for example, \cite{ATW, BR, dc, Dj, dhs, DGL, El, GM,  OFa, Sa, Sh1, ST,  Wo}, or the monograph \cite{IS}. 

An element $ g$  of a group $ \mathcal G$  is called \emph{reversible} or \emph{real} if it is conjugate in $\mathcal G$  to its inverse; equivalently,
there exists $ h\in \mathcal G$  such that $hgh^{-1}=g^{-1}$. 
If the conjugating element $h$  can be chosen as an involution, i.e.,~$h^2=e~(\text{the identity element of}~\mathcal{G})$, then $ g$ 
is called \emph{strongly reversible} or \emph{strongly real}. Equivalently, strongly reversible elements can be expressed as the product of two involutions. In arithmetic contexts, the term \emph{reciprocal} has also been used for strongly reversible elements, see \cite{bv, dg, Sa}. 

Let $G:={\rm Aff}(\mathbb T^{2})$ denote the group of affine toral automorphisms of the $2$-torus $\mathbb T^{2}$. That is, 
\[ G:= {\rm Aff}(\mathbb T^{2})
   = GL(2,\mathbb Z)\ltimes \mathbb T^{2},\qquad
(A,\overline{a})\cdot(\bar{x})=A\bar{x}+\overline{a} \pmod{\mathbb Z^{2}}. 
\]
This group provides one of the simplest nontrivial examples in which linear, algebraic, and
dynamical aspects interact in a rich way.  The linear part $ A\in GL(2,\mathbb Z)$ 
governs the global qualitative behaviour of the map while the translation part $ \overline{a}\in\mathbb T^{2}$  accounts for the finer geometric structure.  

Topological conjugacy and reversibility in the group of linear automorphisms of the $2$-torus $\mathbb{T}^{2}$ have been explored in \cite{ATW, BR2}, see \cite{IS} for an exposition. Recall that the $2$-torus can be identified as  $\mathbb{T}^2 = \mathbb{R}^2 / \mathbb{Z}^2$. Every toral automorphism is induced by a linear map of $\mathbb{R}^2$ preserving the lattice $\mathbb{Z}^2$, and hence corresponds to a matrix in $\mathrm{GL}(2,\mathbb{Z})$. The induced action $\bar{x} + \mathbb{Z}^2 \mapsto A\bar{x} + \mathbb{Z}^2$ defines a diffeomorphism of $\mathbb{T}^2$, yielding a natural identification of the group of toral automorphisms with $\mathrm{GL}(2,\mathbb{Z})$. 

\medskip In this paper, we address analogous questions for affine $2$-toral automorphisms. Baake and Roberts~\cite{BR2} considered this problem in the context of reversing symmetries of affine toral automorphisms. A systematic classification, however, is still lacking, as the affine case presents additional subtleties. In particular, as we shall see, the (strong) reversibility of $(A,\overline{a})$ need not coincide with that of its linear part, and depends crucially on whether $1$ is an eigenvalue of $A$.

\medskip  The purpose of this paper is twofold.  First, we give a systematic description
of reversibility and strong reversibility in the affine group
$ GL(2,\mathbb Z)\ltimes\mathbb T^{2}$ .  We identify the precise algebraic
obstructions that come from the translation component and express them in explicit
linear congruence conditions.  In doing so, we show that the behavior of reversibility in the affine setting depends according to the eigenspace structure of $ A$, with a particularly transparent description in the hyperbolic case.

 Second, we clarify several geometric aspects of affine toral dynamics that
interact with reversibility.  In particular, we establish a sharp fixed-point
criterion for the affine map $f_{A,\overline{a}}: \bar{x}\mapsto A\bar{x}+\overline{a}~(\text{mod}~\mathbb{Z}^2)$, based on Pick's Theorem \cite{GP} and a
detailed analysis of the lattice geometry of the image of the fundamental square under $ A-I$.  While the existence of a fixed point for $f_{A,\overline{a}}(\bar{x})=A\bar{x}+\overline{a}~(\text{mod}~\mathbb{Z}^2)$ follows directly from the invertibility of $A-I$, this topological argument reveals neither its location, the number of distinct lifts, nor the influence of the geometry of $A$ on these features. The Pick--type criterion in Theorem~\ref{fixed-point} serves a different and complementary purpose. The theorem asserts the existence of a fixed point by explicitly ensuring that the parallelogram $(A-I)[0,1)^2$ contains an interior lattice point.  This avoids topological machinery, works directly with the arithmetic structure of $A$, and moreover yields quantitative information, such as bounds on the number of
possible lifts of a fixed point. It provides an accessible geometric tool that is useful when one seeks explicit control over fixed points rather than merely their abstract existence.

\medskip We offer a $\gcd$-based boundary lattice count which  gives a fixed point criterion as follows. 

For an element $A= \begin{pmatrix}
a & b\\ c& d    
\end{pmatrix}$ in $GL(2, \mathbb Z)$ without an eigenvalue $1$, the affine map $\bar{x} \mapsto A\bar{x} + \overline{a}\pmod{\mathbb{Z}^2}$
has a fixed point if 
$$\det(A-I) \ge \gcd(a-1,c) + \gcd(b,d-1).$$

In addition to this, we describe the structure of conjugacy and similarity classes in
$\mathrm{GL}(2,\mathbb{Z}) \ltimes \mathbb{T}^{2}$. If $\det(A-I) \neq 0$, each similarity class contains 
only finitely many affine conjugacy classes; by contrast, when $1$ is an eigenvalue 
of $A$, a similarity class contains uncountably many.

\medskip The main  results of the paper can be summarized as follows:
\begin{itemize}

\medskip \item We establish a  geometric fixed-point criterion for affine toral maps using Pick’s theorem and explicit lattice-point counts, cf. Theorem~\ref{fixed-point}. Moreover, we show that these maps have the same topological entropy as the corresponding induced toral automorphisms, cf. Theorem~\ref{top entrpy}.

  \medskip \item We describe in detail the orbit structure of the translation components under the action of $GL(2,\Z)$, and determine the number of conjugacy classes contained in a given similarity class, cf.  Theorem~\ref{orbit counting}.

\medskip \item We give explicit necessary and sufficient conditions for reversibility and
      strong reversibility of affine toral maps $ (A,\overline{a})$, depending on the
      conjugacy type of $ A$, the solvability of linear congruences involving
      $ A,~ \overline{a}$ and possible reversing involutions, cf. Theorem~\ref{revtoraut}. 

    \end{itemize}

From this work, it is observed that the arithmetic structure of $\mathrm{GL}(2,\mathbb{Z})$, 
the linear action of $A$ on the torus, and the position of~$\overline{a}$ in $\mathbb{T}^2$ 
together give rise to interesting dynamics. Despite its apparent complexity, all 
calculations can be carried out using basic tools. From a dynamical viewpoint, the reversibility conditions for affine toral automorphisms
can be interpreted in terms of a linear cohomological equation over the base toral
automorphism, with reversibility corresponding to the absence of a cohomological
obstruction; this perspective is discussed briefly at the end of the paper, cf. Section~\ref{coho}. 
Hopefully, the methods introduced 
here will be useful for studying reversible dynamics on higher-dimensional tori and their 
associated (affine) groups.

In summary, this paper shows that affine toral automorphisms display new arithmetic and cohomological obstructions to reversibility, new geometric fixed-point criteria, and an interesting finite versus uncountable conjugacy dichotomy, depending on 1 is not an eigenvalue or otherwise,  that does not occur in the linear setting.

\subsection*{Structure of the paper}

Following the introduction, Section~\ref{prel} and Section~\ref{glz} are mostly preparatory. In Section~\ref{prel}, we recall preliminary notions. In Section~\ref{glz}, We recall the reversible and strongly reversible elements of $GL(2,\mathbb{Z})$ and $PGL(2,\mathbb{Z})$. 

Section~\ref{fp} establishes a fixed-point criterion for affine toral automorphisms via Pick’s theorem and also provides some other dynamical properties. In Section~\ref{number}, we determine the number of conjugacy classes within a given similarity class of the group $GL(2,\mathbb{Z})\ltimes\mathbb{T}^2.$ Finally, Section~\ref{affine}  provides a necessary and sufficient criterion for reversibility and strong reversibility of the affine toral automorphism $f_{A,\overline{a}}$.

\textbf{Notations.} For any matrix \(A\), we denote by \(\operatorname{tr}(A)\) its trace and by \(\det(A)\) its determinant. The image of the linear map induced by \(A\) is denoted by \(\operatorname{Im}(A)\). We write \(M(2,\mathbb{Z})\) for the group of all \(2\times 2\) matrices with integer entries and
\[
GL(2,\mathbb{Z})=\{A\in M(2,\mathbb{Z}) : \det(A)=\pm 1\}.
\]
The groups \(GL(2,\mathbb{Q})\) and \(PGL(2,\mathbb{Z})\) have their usual meaning.

\section{Preliminaries}\label{prel}
Let $\mathbb{T}^2=\mathbb{R}^2/\mathbb{Z}^2$ be the $2$-torus, viewed as an abelian group under addition. The group of affine toral automorphisms of $\mathbb{T}^2$ can be identified with the semidirect product $GL(2,\mathbb Z)\ltimes\mathbb T^{2}$. 

Throughout this paper, the symbol $G$ will refer to the semidirect product $GL(2,\mathbb{Z})\ltimes\mathbb{T}^2.$

The group operation is defined as follows:
\[(A,\overline{a})(B,\overline{b})=(AB,A\overline{b}+\overline{a}~(\text{mod}~\mathbb{Z}^2)), ~\text{for all}~(A,\overline{a}),(B,\overline{b})\in G.\]

Then the inverse of an element $(A,\overline{a})\in G$ is $(A^{-1},-A^{-1}\overline{a} ~(\text{mod}~\mathbb{Z}^2))$. 

\subsection{Affine reversibility}

We work in the affine semidirect product $G$, where an element \((A,\overline{a})\in G\) acts on the torus \(\mathbb{T}^2\)
by the affine map
$$
f_{A,\overline{a}} : \mathbb{T}^2 \longrightarrow \mathbb{T}^2,\qquad
f_{A,\overline{a}}(\bar{x})=A\bar{x}+\overline{a} ~(\text{mod}~\mathbb{Z}^2).
$$

Each element $(A,\overline{a})$ of $G$ will induce a dynamical system $(f_{A,\overline{a}},\mathbb{T}^2).$ The study of periodic points of a dynamical system is important because these are the simplest orbit structures we can have. The number of periodic points can also be used to determine the topological entropy of a dynamical system, which is the measure of the exponential growth rate of the number of distinguishable orbits, see \cite{BS}. 

A point $x$ is called a fixed point of a map $f$ if it satisfies $f(x)=x$. A periodic point is essentially a fixed point of some higher iterate of the map. Here we are interested only in the study of fixed points of the affine map $f_{A,\overline{a}} $ in a topological and geometrical viewpoint. In this case, we apply Pick’s Theorem \cite{GP}.

\begin{theorem}[Pick's Theorem]
Let the vertices of a polygon $S$ have all its co-ordinates as integers. Then the area of $S$ is given by \[\text{Area}(S)=k_1+\frac{k_2}{2}-1,\] where $k_1, k_2$ denote respectively the number of interior and boundary (including both the vertices and points along the sides) points with integer co-ordinates.     
\end{theorem}

\begin{remark}

The proof of Pick’s theorem is illustrated geometrically in {Figure~\ref{fig: grid}. At the center of each unit square there is a black lattice point, and a polygon is drawn by joining seven selected lattice points. The resulting closed polygonal curve is shown in blue. Each unit square is coloured according to its geometric relation with the polygon.

   \begin{figure}[!htbp]
    \centering
    \includegraphics[scale=.48]{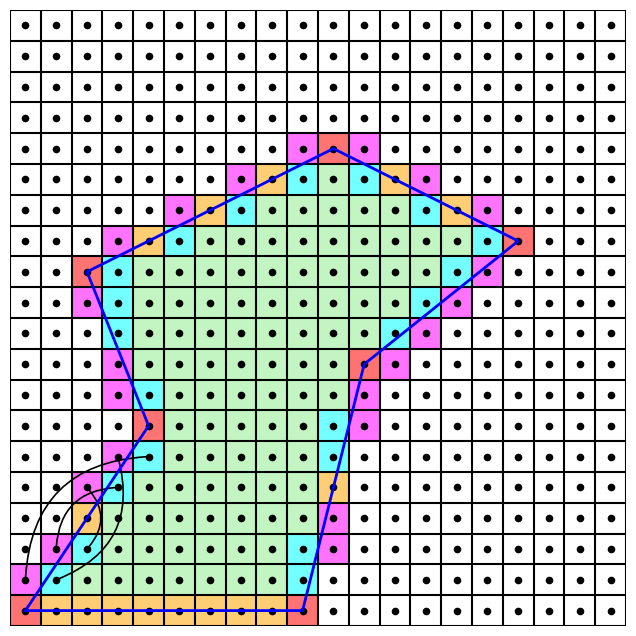}
    \caption{\small Illustration of Pick's Theorem}
    \label{fig: grid}
\end{figure}
\begin{itemize}

    \item squares lying completely inside the polygon are shaded green;

    \item squares intersecting the boundary but whose midpoint lies inside the polygon are shaded light blue;
      
    \item squares intersecting the boundary but whose midpoint lies outside the polygon are shaded magenta;
    \item squares whose midpoint lies exactly on an edge of the polygon are coloured orange;

    \item squares whose midpoint is a polygon vertex are shaded red.
    
    \end{itemize}

 Every orange square contributes $\frac{1}{2}$ unit area. Each light blue square can be paired with one magenta square, in such a way that the pair in combination contributes one unit to the area of the polygon. As an example, four pairs of  squares are  linked using  black arcs in the figure along one boundary segment of the polygon.  
Each arc joins a magenta square to a light–blue square, illustrating this
interior--exterior pairing  contributing one unit area to the polygon.
}
\end{remark}

 Let us now recall the notion of elliptic, parabolic and hyperbolic matrices.  

\begin{definition}
For $A \in GL(2, \mathbb Z)$, let  
$$c(A)=\frac{{\rm tr}^2(A)}{\det A}.$$
    An element $A\in GL(2,\mathbb{Z})$ is called elliptic if $|c(A)|<4$, and parabolic if $|c(A)|= 4$. For all other values of 
 $c(A)$, the corresponding element $A$ is called hyperbolic.
\end{definition}
A non-hyperbolic matrix $A$ is, by definition, either elliptic or parabolic.

\section{Reversibility in $GL(2,\mathbb Z)$ and $PGL(2,\mathbb Z)$}\label{glz}

In this section we briefly review the classification of reversible and strongly reversible elements in \(GL(2,\mathbb{Z})\) and \(PGL(2,\mathbb{Z})\). Essentially all of the results presented here are well known and can be found, in some form, in the book \cite{IS}; they are included solely to fix notation and for later reference.

\begin{lemma}[{\cite[Lemma~3.2]{IS}}]\label{GL2Z-involutions} Every involution in $GL(2,\mathbb Z)$ is conjugate to one of
\[
-I_2,\qquad 
\begin{pmatrix}0&1\\1&0\end{pmatrix},\qquad
\begin{pmatrix}1&0\\0&-1\end{pmatrix}.
\]
\end{lemma}

\begin{theorem}[{\cite[Theorem~4.1]{IS}}]\label{GL2Z-reversible}An element $A\in GL(2,\mathbb Z)$ is reversible if and only if either
\begin{enumerate}
\item $\det(A)=-1$ and $A$ is an involution, or
\item $\det(A)=1$ and $A$ is conjugate to a matrix
\(
\begin{pmatrix}a&b\\c&d\end{pmatrix}
\)
with $a=d$ or $b=\pm c$.
\end{enumerate}
\end{theorem}

\begin{remark}
If $A\in GL(2,\mathbb Z)$ is reversible (resp.\ strongly reversible), then so is $-A$.
Hence it suffices to consider matrices with nonnegative trace.
\end{remark}

\subsection{Reversibility in $PGL(2,\mathbb Z)$}

Identify the boundary $\mathbb{S}^{1}$ (the unit circle in $\mathbb{R}^2$ centered at the origin) with $\mathbb{R}\cup\{\infty\}$ via projective coordinates. Now consider the involutions:
$$
\mathbf{I}_{1}(z)=\frac{1}{z}, 
\qquad 
\mathbf{I}_{2}(z)=-\frac{1}{z}.
$$

\begin{definition}
Two points $p_{1},p_{2}\in \mathbb{S}^{1}$ form a \emph{reciprocal pair} if the unordered set 
$\{p_{1},p_{2}\}$ is invariant under $\mathbf{I}_{1}$ or $\mathbf{I}_{2}$.  
Equivalently,
$$
\mathbf{I}_{j}(\{p_{1},p_{2}\})=\{p_{1},p_{2}\},\qquad j=1,2.
$$
\end{definition}

\begin{lemma}\label{PGLrev}
A hyperbolic element of $PGL(2,\mathbb Z)$ is reversible if and only if its two fixed
points on $\mathbb S^1$ form a reciprocal or symmetric pair (i.e. invariant under $z\mapsto -z)$.
\end{lemma}
\begin{proof}
Assume $g$ is reversible.  
Then there exists $R\in PGL(2,\mathbb{Z})$ with
$$
RgR^{-1}=g^{-1}.
$$
The two fixed points $\{p_{1},p_{2}\}$ of $g$ are exchanged by $g^{-1}$.  
The conjugacy relation implies that $R$ preserves this set; hence $\{p_{1},p_{2}\}$ is invariant under an involution of $PGL(2,\mathbb{Z})$.  
Every involution is conjugate to $z\mapsto z^{-1}$ or $z\mapsto -z^{-1}$, so the fixed points must be reciprocal or symmetric.

Conversely, suppose the fixed points are reciprocal.  
If $p_{1}p_{2}=\pm1$, then solving the fixed-point equation
$$
\frac{az+b}{cz+d}=z
$$
yields
$$
p_{1}+p_{2}=\frac{a-d}{c},\qquad
p_{1}p_{2}=-\frac{b}{c}.
$$
Thus $b=\pm c$, so $g$ is conjugate to one of
$$
\begin{pmatrix} a & b \\ -b & d\end{pmatrix},\qquad
\begin{pmatrix} a & b \\ b & d\end{pmatrix},
$$
each reversed by an involution of $PGL(2,\mathbb{Z})$.  
The symmetric case is identical.
\end{proof}

\subsection{Strong reversibility}

\begin{theorem}[{\cite[Theorem~5.3]{IS}}]\label{GL2Z-strong}
Let $A\in GL(2,\mathbb Z)$.
\begin{enumerate}
\item If $A$ is elliptic or parabolic, then $A$ is strongly reversible.
\item If $A$ is hyperbolic, then $A$ is strongly reversible if and only if it is conjugate
to a matrix
\[
\begin{pmatrix}a&b\\c&d\end{pmatrix}
\quad\text{with}\quad a=d \text{ or } b=-c.
\]
\end{enumerate}
\end{theorem}

\section{Fixed Points and Basic Dynamical Properties}\label{fp} 
The main focus of this section is a geometric fixed-point criterion for affine toral automorphisms.
In addition, we compute the topological entropy of affine toral automorphisms, relating it to the classical computation for linear toral automorphisms, see \cite[Chapter 5]{PY}.

\medskip
We begin with the following standard lemma; see Theorem 10.14(3) of \cite{JH}. More information about the dynamics of toral automorphisms can be found in \cite{PY, BS}.

\begin{lemma}\label{onto endomorphism}
Let $A \in M(2,\mathbb{Z})$ with $\det(A)\neq 0$.  
Then the toral automorphism $f_A$ given by
$$
\bar{x} \mapsto A\bar{x}~(\text{mod}~\mathbb{Z}^2)
$$
is a surjective self–map on $\mathbb{T}^2$ with the topological degree $|\det(A)|$  
$($e.g., see $Figure~\ref{fig: proj})$.
\end{lemma}

\begin{figure}[h]
    \centering
    \includegraphics[scale=.5]{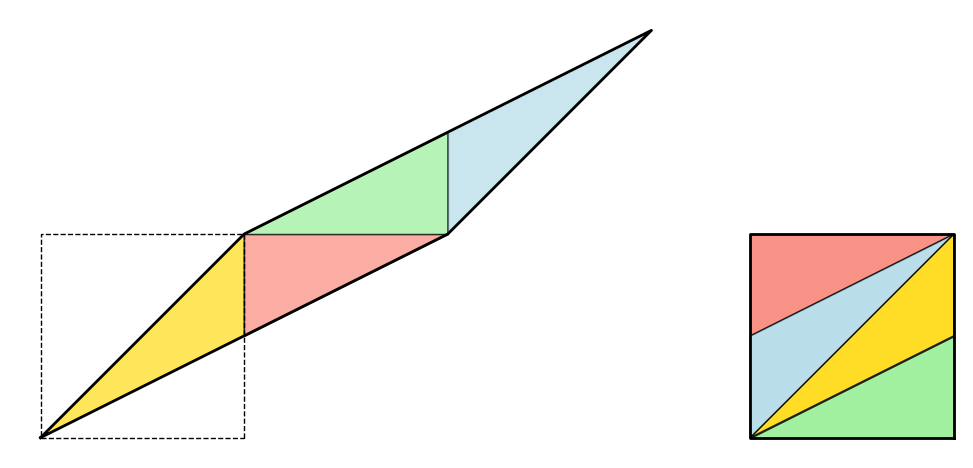}
    \caption{\small Toral endomorphism for the Arnold's cat map induced by $A=\begin{pmatrix}
    2 & 1\\ 1 &1
\end{pmatrix}$. Left figure depicts the image $A([0,1]^2)$ in $\mathbb{R}^2 $, while the right figure depicts the same image after taking $\text{mod}~ \mathbb{Z}^2$.}
    \label{fig: proj}
\end{figure}

\begin{definition}\label{topological conjugacy}
    Two dynamical systems $(X,f)$ and $(Y,g)$ are said to be topologically conjugate if there exists a homeomorphism $h:X\to Y$ such that $h\circ f=g\circ h$. In particular, if the following diagram commutes:
\end{definition}

\[ \begin{tikzcd}
X\arrow{r}{f} \arrow[swap]{d}{h} & X \arrow{d}{h} \\%
Y \arrow{r}{g}& Y
\end{tikzcd}
\]

When two dynamical systems are topologically conjugate, their dynamics on the respective spaces are similar, and the conjugacy $h$ maps the orbit of one system to the orbit of the other.

\subsection*{Affine toral maps and conjugacy}

Every element $(A,\overline{b}) \in G$ induces the affine toral map
$$
f_{A,\overline{b}}(\bar{x}) = A\bar{x} + \overline{b} ~(\text{mod}~\mathbb{Z}^2).
$$
A natural question is whether these affine maps are dynamically similar to the linear toral automorphisms
$$
\bar{x} \mapsto A\bar{x}~(\text{mod}~\mathbb{Z}^2).
$$
When $1$ is not an eigenvalue of $A$, the answer is affirmative.

\begin{proposition}\label{conjugacy}
Let $A \in GL(2,\mathbb{Z})$ with $1$ not an eigenvalue.  
Then the maps $f_{A,\overline{b}}: \bar{x} \mapsto A\bar{x} + \overline{b}~(\text{mod}~\mathbb{Z}^2)$ and $f_A:\bar{x} \mapsto A\bar{x}~(\text{mod}~\mathbb{Z}^2)$ are topologically conjugate on $\mathbb{T}^2$ via a translation.
\end{proposition}

\begin{proof}
Since $1$ is not an eigenvalue of $A$, the equation $\bar{v} = A\bar{v} + \overline{b}$ has a unique solution $\bar{v} \in \mathbb{T}^2$.  
The translation $\bar{x} \mapsto \bar{x}+\bar{v}~(\text{mod}~\mathbb{Z}^2)$ conjugates the affine map to the linear map.
\end{proof}

\begin{remark}
Proposition~\ref{conjugacy} fails when $1$ is an eigenvalue of the matrix $A$.
Indeed, let
\[
A=\begin{pmatrix}
1 & 1\\
0 & 1
\end{pmatrix}
\quad \text{and} \quad
\bar{b}=
\begin{pmatrix}
0\\
\alpha
\end{pmatrix},
\]
where $\alpha\in(0,1)\setminus\mathbb{Q}$.
Then the toral automorphism
\[
\bar{x} \longmapsto A\bar{x} \pmod{\mathbb Z^{2}}
\]
and the affine toral automorphism
\[
\bar{x} \longmapsto A\bar{x}+\bar b \pmod{\mathbb Z^{2}}
\]
are not topologically conjugate.
\end{remark}

We compute the topological entropy of the affine toral automorphisms next.  

\begin{theorem}\label{top entrpy}
   Let $A \in GL(2,\mathbb{Z})$ and $\bar{a} \in \mathbb{T}^2$. Then the toral automorphism $f_A$ and the affine toral automorphism $f_{A,\bar{a}}$ have the same topological entropy; that is,
\[
h_{\mathrm{top}}(f_A) = h_{\mathrm{top}}(f_{A,\bar{a}}).
\]
 
\end{theorem}
\begin{proof}
If $1$ is not an eigen value of $A,$ then the proof follows from Proposition \ref{conjugacy}.    

Now assume 1 is an eigenvalue of $A \in GL(2,\mathbb{Z})$. This means the other eigenvalue is either $1$ or $-1$. If both eigenvalues of $A$ are equal to $1$ then there exists an invertible matrix $P$ so that, $PAP^{-1}=I+N, ~N^2=O \in M(2, \mathbb Z)$ and hence $PA^nP^{-1}=I+nN \implies ||PA^nP^{-1}|| \leq 1+n||N||=1+n$, where $||.||$ denotes the operator norm of a matrix. This means $d(f_A^n(x),f_A^n(y))\leq nC \cdot d(x,y)$ ($C$ is independent of $n$), where $d$ is the induced metric on the the torus.

Recall the Bowen metric $d_n$ on a compact metric space $X$, which measures the distance between orbits of length $n$,  $$d_n(x, y) = \max_{0 \leq i < n} d(f_A^i(x), f_A^i(y)).$$ Using the Lipschitz condition on each iterate $f_A^i$, we have,  $d(f_A^i(x), f_A^i(y))  \leq iC  d(x, y)$ and hence for $i \leq n$, we get that $d_n(x, y) \leq nC  d(x, y)$.

Let $N_n(\mathbb{T}^2, \epsilon)$ be the minimum number of $\epsilon$-balls in the $d_n$ metric required to cover $\mathbb{T}^2$. From the inequality above, we see that a ball in the original metric $d$ with radius $ \frac{\epsilon}{Cn}$ is entirely contained within $\epsilon$-ball in the $d_n$ metric. Therefore, the number of balls needed to cover $\mathbb{T}^2$ in the $d_n$ metric is at most the number of balls needed to cover $\mathbb{T}^2$ in the original metric with radius $\frac{\epsilon}{Cn}$ i.e., $N_n(\mathbb{T}^2, \epsilon) \leq N\left(\mathbb{T}^2, \frac{\epsilon}{Cn}\right),$ where $N(\mathbb{T}^2, \delta)$ is the minimum number of balls of radius $\delta$ required to completely cover   $\mathbb{T}^2$ in the $d$ metric. Clearly, $N(\mathbb{T}^2, \delta)=O(1/\delta^2)$, which means, 
\[
\begin{array}{c}
N_n(\mathbb{T}^2, \epsilon)
\leq N\!\left(\mathbb{T}^2, \dfrac{\epsilon}{Cn}\right)
\leq \dfrac{C_0 n^2}{\epsilon^2}. \\[0.6em]
\end{array}
\]

This yields, 

\[
\begin{array}{c}
\dfrac{\log N_n(\mathbb{T}^2, \epsilon)}{n}
\leq
\dfrac{\log C_0 + 2\log n - 2\log \epsilon}{n}
\end{array}
\]

hence,  
\[
\begin{array}{rcl}
\displaystyle
\limsup_{n \to \infty} \frac{1}{n}\log N_n(\mathbb{T}^2,\epsilon)
& \le &
\displaystyle
\limsup_{n \to \infty} \frac{\log C_0 + 2\log n - 2\log \epsilon}{n} \\[0.8em]
& = &
\displaystyle
\lim_{n \to \infty} \frac{\log C_0 + 2\log n - 2\log \epsilon}{n} \\[0.8em]
& = & 0 .
\end{array}
\]

Therefore, the topological entropy of the toral automorphism $f_A$ is
\[
h_{\mathrm{top}}(f_A)
=
\lim\limits_{\epsilon \to 0}
\limsup_{n \to \infty}
\frac{1}{n}\log N_n(\mathbb{T}^2,\epsilon)
= 0 .
\]

If the other eigenvalue is $-1$ then $A$ is conjugate to 
$\begin{pmatrix}
1 & 0\\
0 & -1
\end{pmatrix}$ so that $||A^n||\leq C$. It is clear that the same argument as above applies to this case. Hence $h_{\mathrm{top}}(f_{A})=0$, if one of the eigenvalues is equal to $1$. 

Since $f_{A,\bar{a}}$ is just a translation of $f_A$ by $\bar{a}$, the Lipschitz bound on $d_n$ metric remains unchanged and the whole proof above applies. Hence, we can conclude that $h_{\mathrm{top}}(f_A)=h_{\mathrm{top}}(f_{A,\bar{a}})=0$, if one of the eigenvalues is equal to 1.

\end{proof}

\subsection{Fixed-point criterion}

We now give a geometric condition guaranteeing that an affine toral map possesses a fixed point.
Let $ A=\begin{pmatrix}a&b\\ c&d\end{pmatrix}\in GL(2,\mathbb Z)$  and then
$$
A-I=\begin{pmatrix}a-1 & b\\ c & d-1\end{pmatrix}.
$$
Let $ v=(a-1,c)^{T}$  and $ w=(b,d-1)^{T}$  denote the column vectors given by
$ (A-I)e_1,(A-I)e_2, \text{where} ~e_1, e_2$ are the standard basis vectors of $\mathbb{R}^2.$ The parallelogram
$$
\mathcal P := (A-I)([0,1)^2) \subset\mathbb R^2
$$
is a fundamental domain (a choice of representatives) for the image $ (A-I)(\mathbb T^2)$ 
viewed in $ \mathbb R^2$  modulo $ \mathbb Z^2$. Its area equals
$$
\operatorname{Area}(\mathcal P)=\big|\det(A-I)\big| = \big|\det(A)-\operatorname{tr}(A)+1|,
$$
where the last equality follows from expanding $ \det(A-I)$.

The following elementary lattice-count fact will be used repeatedly.

\begin{lemma}\label{lem:segment-gcd}
Let $ u=(u_1,u_2)\in\mathbb Z^2$  be a nonzero integer vector. The number of lattice
points on the closed segment joining $ (0,0)$  and $ u$  equals $ \gcd(u_1,u_2)+1$.
\end{lemma}

\begin{proof}
By elementary lattice geometry: the points on the segment correspond to integer
multiples $ \frac{k}{\ell}u$  with $ 0\le k\le\ell$  where $ \ell$  is the largest
integer so that $ \frac{1}{\ell}u\in\mathbb Z^2$. That largest $ \ell$  is
$ \gcd(u_1,u_2)$.
\end{proof}

\begin{proposition}
\label{propb}
With the notation above, the number $ B$  of lattice points on the boundary
$ \partial\mathcal P$  of the parallelogram $ \mathcal P$  equals
$$
B \;=\; 2\big(\gcd(a-1,c)+\gcd(b,d-1)\big).
$$
\end{proposition}

\begin{proof}
The four sides of $ \mathcal P$  are parallel to the segments determined by
$ v$  and $ w$. By Lemma \ref{lem:segment-gcd}, each segment parallel to $ v$ 
contains $ \gcd(a-1,c)+1$  lattice points (including its two endpoints), and each
segment parallel to $ w$  contains $ \gcd(b,d-1)+1$  lattice points. Summing over
the four sides gives $ 2(\gcd(a-1,c)+\gcd(b,d-1))+4$  counts, but each of the
four vertices was counted twice in this sum; subtracting the overcount yields
the asserted formula $ B=2(\gcd(a-1,c)+\gcd(b,d-1))$.
\end{proof}

The next theorem gives an arithmetic condition on the existence of fixed points. 

\begin{theorem}\label{fixed-point}
Let $ A=\begin{pmatrix}a&b\\ c&d\end{pmatrix}\in GL(2,\mathbb Z)$, and assume
that $ 1$  is not an eigenvalue of $ A ~(\text{equivalently,}~ \det(A-I)\neq 0)$. Let
$$
g_1:=\gcd(a-1,c),\qquad g_2:=\gcd(b,d-1).
$$
If
\begin{equation}\label{eqarea}
\det(A-I) \;\ge\; g_1+g_2,
\end{equation}
then for every translation vector $ \overline{b}\in\mathbb T^2$,  the affine map
$ f_{A,\overline{b}}(\bar{x})=A\bar{x}+\overline{b}~(\text{mod}~\mathbb{Z}^2)$  has a fixed point in $ \mathbb T^2$.
\end{theorem}

\begin{proof}
A fixed point $ \bar{x}\in\mathbb T^2$  for $ f_{A,\overline{b}}$  satisfies the congruence
$$
(A-I)\bar{x} \equiv -\overline{b} \pmod{\mathbb Z^2},
$$
i.e., there exists an integer vector $ \zeta\in\mathbb Z^2$  and a representative
$ \tilde x\in[0,1)^2\subset\mathbb R^2$  with
$$
(A-I)\tilde x = \zeta - \overline{b_0},
$$
where $ \overline{b_0}$  is a chosen lift of $ \overline{b}$  to $ \mathbb R^2$. Thus $ -\overline{b}$  has a
representative in the image $ (A-I)([0,1)^2)=\mathcal P$  if and only if the
map $ f_{A,\overline{b}}$  has a fixed point.

By Proposition \ref{propb}, the number $ B$  of lattice points on
$ \partial\mathcal P$  equals $ 2(g_1+g_2)$. Applying Pick's Theorem to the lattice polygon $ \mathcal P$  (which is a parallelogram with integer vertices) gives the relation
$$
N \;=\; \operatorname{Area}(\mathcal P) - \frac{B}{2} + 1,
$$
where $ N$  is the number of integer points in the interior of $ \mathcal P$.
If $ \operatorname{Area}(\mathcal P)\ge \frac{B}{2}$  then $ N\ge 0+1-0=1$,
so $ \mathcal P$  contains at least one interior lattice point. The inequality
$ \operatorname{Area}(\mathcal P)\ge \frac{B}{2}$  is equivalent to
$ \det(A-I)\ge g_1+g_2$, which is exactly \eqref{eqarea}.

An interior lattice point $ p\in\mathcal P\cap\mathbb Z^2$  gives $ \zeta=p$ 
and hence provides a solution $ \tilde x\in[0,1)^2$  of $ (A-I)\tilde x=\zeta-\overline{b_0}$,
so $ -\overline{b}$  is represented in $ \mathcal P$  and the affine map $ f_{A,\overline{b}}$  has a
fixed point. 

\end{proof}

\begin{remark}

A geometric visualization of the proof of Theorem \ref{fixed-point} is shown in Figure~\ref{fig: rep}.  
If the interior lattice point lies fully inside a unit square, we are done immediately.  
If it lies in a triangular region along the boundary, one can translate this triangle by a lattice vector into the parallelogram, forming a full unit square and thus a representative for $-\overline{b}$. 

\begin{figure}[hbt!]
    \centering
    \includegraphics[scale=.6]{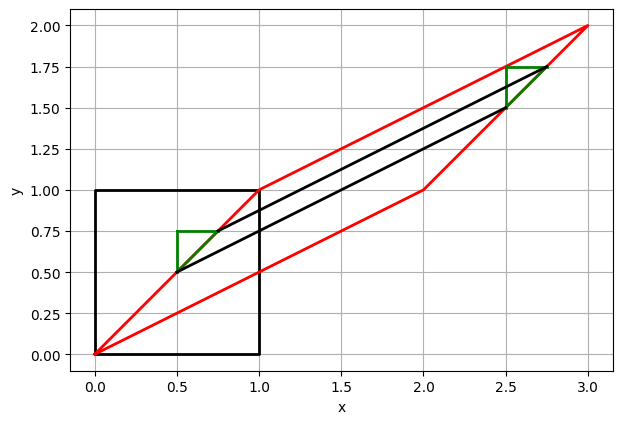}
    \caption{\small Recovering a representative of $-\overline{b}$ inside the parallelogram.}
    \label{fig: rep}
\end{figure}
\end{remark}

\begin{remark}
 Since $ \det(A-I)=\det(A)-\operatorname{tr}(A)+1$, the inequality
      \eqref{eqarea} can be rewritten entirely in terms of $ \det A$ 
      and $ \operatorname{tr}A$  together with the two gcds $ g_1,g_2$.
\end{remark}

\begin{remark}
If 1 is not an eigenvalue of $A$ then the affine map $ f_{A,\overline{b}}$ is conjugate to the linear map $f_A$ and this conjugacy does not depend on the fact whether any component of the translation vector $\overline{b}$ is irrational or not. In a way the Theorem \ref{fixed-point} gives a little more information on the existence of fixed points directly. This result can also be used to determine the existence of periodic points of an affine map by iterating the original one, given the iteration satisfies the conditions of Theorem \ref{fixed-point}.         
\end{remark}

\begin{example}
    Now we will produce a class of examples  that satisfy the condition of Theorem \ref{fixed-point}. So we can actually  apply the theorem for this class of infinitely many matrices. Consider the sequence of matrices
\[
A_n \;=\;
\begin{pmatrix}
-n & 1 \\[6pt]
1 & 0
\end{pmatrix}
\in GL(2,\mathbb{Z}).
\]
Then
\[
A_n - I
=
\begin{pmatrix}
-n-1 & 1 \\[6pt]
1 & -1
\end{pmatrix},
\]
and therefore
\[
\det(A_n - I)
=n.
\]

Hence
\[
\det(A_n - I) \;=\; n \;\to\; \infty, \text{as} ~n\to \infty.
\]

Next, compute both the gcds:
\[
g_1=\gcd(-n - 1,\, 1)= 1,
\]
\[
g_2=\gcd(1,\, -1)= 1.
\]

Therefore
\[
\det(A_n - I) - g_1- g_2=n - 1 - 1=n - 2,
\]
and consequently
\[
\det(A_n - I) - g_1 - g_2
\;\to\;
\infty, ~\text{as}~n\to \infty.
\]

\end{example}

\begin{remark}
When $A\in  GL(2,\mathbb{Z})$ has $1$ as an eigenvalue, it is possible that $f_{A,\overline{b}}$ has no fixed point for every nonzero $\overline{b}=(b_1,b_2)^T\in \mathbb{T}^2.$ As an example, for the matrix  $
A = \begin{pmatrix} 1 & 1 \\ 0 & 1 \end{pmatrix}
$, the map $f_{A,\overline{b}}$ to have a fixed point, it is necessary that $b_2\equiv 0 ~(\text{mod}~\mathbb{Z}).$ 
\end{remark}

\section{Number of Conjugacy classes in a Similarity Class}\label{number} 
In this section, we focus on the algebraic properties of the group $G$.
It is known from \cite{ATW} that the action of $GL(2,\mathbb{Z})$ on the upper half-plane has only finitely many conjugacy classes inside each similarity class.
A similar phenomenon occurs for the affine action on $\mathbb{T}^2$. First note that by similarity we mean the following. 
\begin{definition}\label{sim} Let $(A,\overline{a})~\text{and}~(B,\overline{b}) ~\text{be elements of}~ G$.  
We say that $(A,\overline{a})$ and $(B,\overline{b})$ are \emph{similar} if there exist 
$C \in GL(2,\mathbb{R})$ and $\overline{c} \in \mathbb{T}^{2}$ such that 
$$(C,\overline{c})(A,\overline{a})(C,\overline{c})^{-1}=(B,\overline{b}).$$  

If the element $C$ can be chosen from  $GL(2,\mathbb{Z})$, then 
$(A,\overline{a})$ and $(B,\overline{b})$ are said to be \emph{conjugate}. We say that $(A,\bar a)\in G$ is conjugate to a \emph{linear map} if it is
conjugate in $G$ to $(A,\bar{0})$.
\end{definition}

We determine the number of conjugacy classes contained in a given similarity class in $G$.  
The next result describes how many conjugacy classes lie in a given similarity class.
This depends entirely on whether the matrix $A \in GL(2,\mathbb{Z})$ has $1$ as an eigenvalue.
\begin{theorem} \label{orbit counting}

Let $(A,\bar a)\in G$. Then the following statements hold:

\begin{enumerate}
    \item If $1$ is not an eigenvalue of $A$, then the similarity
    class of $(A,\bar a)$ contains only finitely many conjugacy classes in $G$.

    \item Assume that $1$ is an eigenvalue of $A$ and $A\neq I$. Then the similarity class of  $(A,\bar a)$ contains uncountably many conjugacy classes.
\end{enumerate}
\end{theorem}

\begin{proof}
We view $G$ inside the affine group $GL(2,\mathbb R)\ltimes \mathbb R^2$
via lifts. We say that $(A,\bar a)$ and $(A',\bar a')$ in $G$ are
$GL(2,\mathbb R)$–similar if there exists $(C,\widehat{c})\in
GL(2,\mathbb R)\ltimes \mathbb R^2$ and lifts $\widehat{a},\widehat{a'}\in\mathbb R^2$ of
$\bar a,\bar a'$ such that
\[
(A',\widehat{a'}) = (C,\widehat{c})^{-1}(A,\widehat{a})(C,\widehat{c})
\]
in $GL(2,\mathbb R)\ltimes\mathbb R^2$.

Conjugation in $G$ by $(B,\bar b)\in G$ is given by
\[
(A,\bar a)\mapsto
\bigl(B^{-1}AB,\; B^{-1}\bar a + B^{-1}(A-I)\bar b\bigr),
\]
where $\bar b\in\mathbb T^2$ is represented by a vector in $\mathbb R^2$ and the
computation is performed modulo $\mathbb Z^2$.

Fix the matrix $A$. Then two elements $(A,\bar a)$ and $(A,\bar a')$ are conjugate in $G$
if and only if there exist $B\in Z_{GL(2,\mathbb Z)}(A)$, the centralizer of $A$ in the group $GL(2,\Z)$ and
$\bar b\in\mathbb T^2$ such that
\[
\bar a' = B^{-1}\bar a + B^{-1}(A-I)\bar b .
\]
In particular, modulo the action of the centralizer of $A$,
conjugacy classes in $G$ with linear part $A$ are parametrized by the quotient 
\[
\mathbb T^2 / (A-I)(\mathbb T^2).
\]
\noindent
\textbf{Case 1.} Assume first that $1$ is not an eigenvalue of $A$.
Then $\det(A-I)\neq 0$, so $A-I$ is invertible and induces
a surjective endomorphism of $\mathbb T^2$. Hence
$(A-I)(\mathbb T^2)=\mathbb T^2$, and every $(A,\bar a)$ is conjugate
to $(A,\bar 0)$. Conjugacy classes in the similarity class are therefore determined
only by the integer conjugacy class of the matrix $A$, and from \cite{ATW} there are only finitely many such classes.

\medskip
\noindent
\medskip
\noindent
\noindent
\textbf{Case 2.}
 Recall that  conjugacy classes in $G$ with linear part $A$ are parametrized by the quotient $\mathbb T^2 / (A-I)(\mathbb T^2)$. If $1$ is an eigenvalue of $A$ then the image of $(A-I)$ is at most one dimensional in $\mathbb{R}^2$, upon projection on $\mathbb{T}^2$ this becomes union of at most countably many line segments in the unit square as the  fundamental domain. 
 
 Since $GL(2,\mathbb{Z})$ is countable, the set of translation parameters in $\mathbb{T}^2$ corresponding to elements conjugate to $(A,\bar a)$ is a countable union of at most countable union of line segments, hence a countable union of nowhere dense subsets of $\mathbb{T}^2$. By the Baire category theorem this union cannot cover $\mathbb{T}^2$, so its complement is uncountable. Therefore the similarity class of $(A,\bar a)$ contains uncountably many conjugacy classes in $G$.\end{proof}

\section{Reversible and Strongly Reversible Elements of $GL(2,\mathbb{Z})\ltimes \mathbb{T}^{2}$}\label{affine}

Reversibility in affine toral dynamics couples the classical notion of matrix reversibility with translation dynamics on the torus.

\begin{definition}
The affine map \(f_{A,\overline{a}}\) is \emph{reversible} if there exists \((R,\overline{r})\in G\) such that
$$
(R,\overline{r})\, (A,\overline{a})\, (R,\overline{r})^{-1}=(A,\overline{a})^{-1}.
$$
If one can choose \((R,\overline{r})\) with \((R,\overline{r})^2=(I,\bar{0})\) (equivalently, \(R^2=I\) and \((R+I)\overline{r}\equiv 0~ (\text{mod}~\mathbb{Z}^2)\)),
then \(f_{A,\overline{a}}\) is \emph{strongly reversible}.
\end{definition}

Our study of the reversibility of affine maps begins with the following lemma.

\begin{lemma}
Let $A,R \in GL(2, \mathbb Z)$ and $\overline{a},\overline{r} \in \mathbb T^2 $.
Assume that $RAR^{-1} = A^{-1}$. 
Then the congruence
\begin{equation}\label{eq}
- R A R^{-1} \overline{r} + R\overline{a} + \overline{r} \equiv -A^{-1} \overline{a} \pmod{\mathbb Z^2}
\end{equation}
is equivalent to
\begin{equation}\label{eq2}
(AR+I)\overline{a} + (A-I)\overline{r} \equiv \bar{0} \pmod{\mathbb Z^2}.
\end{equation}
\end{lemma}

\begin{proof}
Using the conjugacy relation $RAR^{-1} = A^{-1}$, the congruence
\eqref{eq} becomes
$$
- A^{-1} \overline{r} + R\overline{a} + \overline{r} \equiv -A^{-1} \overline{a} \pmod{\mathbb Z^2}.
$$
Bringing all terms to the left-hand side gives
$$
- A^{-1} \overline{r} + R\overline{a} + \overline{r} + A^{-1} \overline{a} \equiv \bar{0} \pmod{\mathbb Z^2}.
$$
$$
\Rightarrow (R\overline{a} + A^{-1}\overline{a}) + (\overline{r} - A^{-1} \overline{r}) \equiv \bar{0} \pmod{\mathbb Z^2} 
$$
This implies
\begin{equation}\label{eqq}
(R + A^{-1})\overline{a} + (I - A^{-1})\overline{r} \equiv \bar{0} \pmod{\mathbb Z^2}.
\end{equation}

Since $A \in GL(2, \mathbb Z)$, multiplication of congruences by $A$
preserves equivalence modulo $\mathbb Z^2$.
Multiplying \eqref{eqq} on the left by $A$ yields
$$
A(R + A^{-1})\overline{a}+ A(I - A^{-1})\overline{r} \equiv \bar{0} \pmod{\mathbb Z^2}.
$$
Thus,
$$
AR\overline{a} + \overline{a} + A\overline{r} - \overline{r} \equiv \bar{0} \pmod{\mathbb Z^2}.
$$
Rearranging gives
$$
(AR\overline{a} + \overline{a}) + (A\overline{r} - \overline{r}) \equiv \bar{0} \pmod{\mathbb Z^2}.
$$
So,
$$
(AR + I)\overline{a} + (A - I)\overline{r} \equiv \bar{0} \pmod{\mathbb Z^2}.
$$

Next, we have the congruence $(5).$ Now multiplying $(5)$ on the left by $A^{-1}$ gives
\[R\overline{a}+A^{-1}\overline{a}+\overline{r}-A^{-1}\overline{r}\equiv \bar{0} \pmod{\mathbb Z^2}.\]
 Using $RAR^{-1}=A^{-1},$ we have the congruence $(4).$ Thus, the lemma follows.
\end{proof}

The following theorem describes the reversibility of affine maps in terms of the linear part and the translation part.

\begin{theorem}\label{thm:affine-reversible}
Let $(A,\overline{a})\in G$.  Then \(f_{A,\overline{a}}\) is reversible
if and only if there exists \((R,\overline{r})\in G\) satisfying
$$
RAR^{-1}=A^{-1}
\qquad\text{and}\qquad
(AR + I)\overline{a} + (A - I)\overline{r}\equiv \bar{0} \pmod{\mathbb{Z}^2}.
$$
Consequently:
\begin{enumerate}
  \item If $1$ is \emph{not} an eigenvalue of $A$, then 
$f_{A,\overline{a}}$ is reversible if and only if $A$ is reversible in $GL(2,\mathbb{Z})$.

\item If \(1\) is an eigenvalue of \(A\), then \(f_{A,\overline{a}}\) is reversible
  precisely when there exists \(R\in GL(2,\mathbb{Z})\) with \(RAR^{-1}=A^{-1}\)
  such that the linear congruence
  $$
  (AR+I)\overline{a} \in {\operatorname{Im}}(A-I)\quad\text{$($in }\mathbb{T}^2\text{$)$}
  $$
  holds; equivalently, for that reversing \(R\) there exists \(\overline{r}\in\mathbb{T}^2\)
  satisfying $$(A-I)\overline{r}+(AR+I)\overline{a}\equiv \bar{0} ~(\text{mod}~ \mathbb Z^2)$$
\end{enumerate}
\end{theorem}

\begin{proof}
Compute conjugation in \(G\).  The inverse of \((R,\overline{r})\) is \((R^{-1},-R^{-1}\overline{r}~(\text{mod}~\mathbb{Z}^2))\). Using
\((X,\overline{x})(Y,\overline{y})=(XY,X\overline{y}+\overline{x}~(\text{mod}~\mathbb{Z}^2))\) we get
$$
(R,\overline{r})(A,\overline{a})(R,\overline{r})^{-1}
= \bigl(RAR^{-1},\; -RAR^{-1}\overline{r} + R\overline{a} + \overline{r}~(\text{mod}~\mathbb{Z}^2) \bigr).
$$
Indeed, acting on \(\bar{x}\in\mathbb{T}^2\) gives
$$
(R,\overline{r})\bigl( A(R^{-1}\bar{x} - R^{-1}\overline{r}) + \overline{a} \bigr)
= R A R^{-1} \bar{x} \;-\; R A R^{-1} \overline{r} \;+\; R \overline{a} \;+\; \overline{r},
$$
so the linear part is \(RAR^{-1}\) and the translation part is \(-RAR^{-1}\overline{r} + R\overline{a} + \overline{r}\).

Requiring that this equals the inverse \((A,\overline{a})^{-1}=(A^{-1},-A^{-1}\overline{a}~(\text{mod}~\mathbb{Z}^2))\) gives the two
conditions
$$
RAR^{-1}=A^{-1},
\qquad\text{and}\qquad
- R A R^{-1} \overline{r} + R \overline{a} + \overline{r} \equiv -A^{-1} \overline{a} \pmod{\mathbb{Z}^2}.
$$
Replacing \(RAR^{-1}\) by \(A^{-1}\) in the second equality and rearranging yields
the equivalent linear congruence
$$
(AR + I)\overline{a} + (A - I)\overline{r}\equiv \bar{0}\pmod{\mathbb{Z}^2}.
$$
Thus \(f_{A,\overline{a}}\) is reversible iff there exists \(R\in GL(2,\mathbb{Z})\) with \(RAR^{-1}=A^{-1}\) for which
the above congruence is solvable for \(\overline{r}\in \mathbb{T}^2\).  This proves the main equivalence.

\medskip\noindent
\textbf{Case 1:}
If $1$ is not an eigenvalue of $A$, then $A-I$ is invertible on $\mathbb{R}^2$, hence 
the equation \eqref{eq2} has the unique solution
$$
\overline{r} = -(A-I)^{-1}(AR+I)\overline{a}.
$$
Thus reversibility of $f_{A,\overline{a}}$ is equivalent to existence of $R$ satisfying $RAR^{-1}=A^{-1}$ i.e., reversibility of $A$ in $GL(2,\mathbb{Z})$, and each such 
$R$ lifts uniquely to $(R,\overline{r})$.

\medskip\noindent
\textbf{Case 2:}
If $1$ is an eigenvalue of $A$, then $A-I$ is not invertible.  
Then equivalence \eqref{eq2} has a solution $\overline{r}$ if and only if
$$(AR+I)\overline{a} \in {\rm Im}(A-I).
$$
Thus reversibility holds precisely when there exists $R$ with
$RAR^{-1}=A^{-1}$
and the compatibility condition \eqref{eq2} holds.

Putting all of these together, the theorem follows.
\end{proof}

\begin{corollary}\label{5sr}
 The affine map $f_{A,\overline{a}}$ is strongly reversible if and only if
\begin{enumerate}
    \item $A$ is strongly reversible in $GL(2,\mathbb{Z})$, and
    \item for an involutive reverser $\tau$ with $\tau A\tau^{-1}=A^{-1}$,
    $$
    (A\tau +I)\overline{a}\in\operatorname{Im}(A-I).
    $$
\end{enumerate}
\end{corollary}

\subsection{Classification of Reversible Affine Toral Automorphisms}

Here we present a classification of reversibility for affine toral automorphisms in the next theorem. 

\begin{theorem}\label{revtoraut} 
Let $(A,\overline{a})$ be an element of the group $G$.  
Then the affine map $f_{A,\overline{a}}(\bar{x})=A\bar{x}+\overline{a}~(\text{mod}~\mathbb{Z}^2)$ is reversible or strongly reversible as follows:

\begin{enumerate}
    \item \textit{Elliptic or parabolic case:}  
    If $A$ is elliptic or parabolic, then every affine map $f_{A,\overline{a}}$ is strongly reversible.

    \item \textit{Hyperbolic case:}  
    If $A$ is hyperbolic, then $f_{A,\overline{a}}$ is \emph{reversible} if and only if one of the following conditions is satisfied:
    \begin{enumerate}
        \item $A$ is reversible in $GL(2,\mathbb{Z})$, i.e., there exists a matrix $R \in GL(2,\mathbb{Z})$ such that
        $$
        RAR^{-1}=A^{-1};
        $$
        \item the two fixed points of $A$ on $\mathbb{S}^{1}$ are reciprocal or symmetric;
        \item the image of $\bar{a}$ under the linear map $AR+I$ satisfies the following condition
        $$
        (AR+I)\overline{a} \in \operatorname{Im}(A-I).
        $$
    \end{enumerate}

    \item \textit{Strong reversibility:}  
    $f_{A,\overline{a}}$ is \emph{strongly reversible} if and only if there exists a matrix $R \in GL(2,\mathbb{Z})$ such that
    $$
    RAR^{-1}=A^{-1} \quad \text{and} \quad R^2=I
    $$
    $($i.e., an involutive reversing symmetry$)$, and the following condition is satisfied
    $$
    (AR+I)\overline{a} \in \operatorname{Im}(A-I).
    $$
\end{enumerate}
\end{theorem}

\begin{proof}
The characterisation of reversibility in the affine group is provided by
Theorem \ref{thm:affine-reversible}: the affine map $f_{A,\overline{a}}$ is reversible if and only if there exists
$(R,\overline{r})\in G$ such that
\begin{equation}\label{eq:affrev}
RAR^{-1}=A^{-1}, \qquad (AR+I)\overline{a}+(A-I)\overline{r}\equiv 0 \pmod{\mathbb Z^2}.
\end{equation}
Strong reversibility is characterised by Corollary \ref{5sr}, which states that if
$R^2=I$ then $f_{A,\overline{a}}$ is strongly reversible if and only if
\begin{equation}\label{strong}
(AR+I)\overline{a}\in \operatorname{Im}(A-I).
\end{equation}

We treat the three cases separately.

\subsubsection*{\textbf{Elliptic $A$}.}
In the elliptic case, $A$ has no eigenvalue $1$.
Since every elliptic matrix in $GL(2,\mathbb Z)$ is strongly reversible by the classification in \cite{IS}, we may choose $R\in GL(2,\mathbb Z)$ with $RAR^{-1}=A^{-1}$ and $R^2=I$.
Because $A-I$ is surjective, the congruence~\eqref{strong} always has a
solution $\overline{r}\in\mathbb T^2$.  
Hence every affine map with elliptic linear part is strongly reversible.

\subsubsection*{\textbf{Parabolic $A$}.}
From \cite{IS}, every parabolic element of $GL(2,\mathbb Z)$ is strongly reversible (and hence reversible) and conjugate to one of the forms $\begin{pmatrix}
1 & m\\
0 & 1
\end{pmatrix}$ or $\begin{pmatrix}
-1 & -m\\
0 & -1
\end{pmatrix},$ where $m\in\mathbb{N}.$ If $A$ is conjugate to $\begin{pmatrix}
-1 & -m\\
0 & -1
\end{pmatrix}$, strong reversibility of the map $f_{A,\overline{a}}$ follows the same argument as in the elliptic case.

If $A$ is conjugate to $\begin{pmatrix}
1 & m\\
0 & 1
\end{pmatrix}$, then we can choose $R=\begin{pmatrix}
1 & 0\\
0 & -1
\end{pmatrix}$. In this case, $\text{Im}(AR+I)$ and $\text{Im}(A-I)$ will generate the same subspace of $\mathbb{R}^2$. Therefore, congruence~\eqref{strong} will be satisfied by the matrix $R.$

\subsubsection*{\textbf{Hyperbolic $A$}.}
Since $A$ is hyperbolic, it has no eigenvalue $\pm1$, and therefore $A-I$ and $A+I$ are invertible over $\mathbb Q$.  In particular,
$A-I:\mathbb T^2\to\mathbb T^2$ is surjective, so the congruence
\eqref{eq:affrev} is solvable in $\overline{r}$ if and only if
$$
(AR+I)\overline{a}\in\operatorname{Im}(A-I).
$$

By Corollary~7.29 of \cite{IS} and Theorem~\ref{PGLrev}, a hyperbolic matrix
$A\in GL(2,\mathbb Z)$ is reversible in $GL(2,\mathbb Z)$ if and only if its two
fixed points on $\mathbb S^1$ are reciprocal or symmetric.  Thus, $f_{A,\overline{a}}$ is reversible if
and only if one of the following conditions is satisfied:
\begin{enumerate}
\item[(i)] $A$ is reversible in $GL(2,\mathbb Z),$

\item[(ii)] The fixed points of $A$ on $\mathbb{S}^1$ are reciprocal
or symmetric, and
\item[(iii)] $(AR+I)\overline{a}\in\mathrm{Im}(A-I)$.
\end{enumerate}
This establishes part~(2).

 For strong reversibility of $f_{A,\overline{a}}$~, we need to find an involution $R\in GL(2,\mathbb Z)$ satisfying $RAR^{-1}=A^{-1}$ and the condition~\eqref{strong}.  As noted above, for hyperbolic $A$ the induced map
$A-I$ is surjective on $\mathbb T^2$, so~\eqref{strong} is equivalent to the fact that there exists $\bar{x}\in\mathbb T^2 $ such that $$(AR+I)\overline{a} - (A-I)\bar{x} \in \mathbb Z^2.$$
Thus, strong reversibility holds when $A$ is strongly reversible in
$GL(2,\mathbb Z)$ and the above condition  is satisfied. This proves~(3).

Taking all all cases together, the theorem follows.
\end{proof}

\begin{remark} For hyperbolic $A$, the images of $A-I$ and $A+I$ coincide as subspaces of $\mathbb{R}^2$, since the eigenvalues $\lambda$ satisfy $\lambda \neq \pm 1$.  
Hence, one can write the condition in the form
$$
(AR+I)\overline{a} \in \operatorname{Im}(A+I),
$$
which is equivalent in the hyperbolic case. 
\end{remark}
\subsection{Cohomological interpretation of the reversibility condition}\label{coho}  \;\

\medskip The congruence
\[
(AR+I)\bar{a} + (A-I)\bar{r} \equiv \bar{0} \pmod{\mathbb{Z}^2},
\]
which characterizes reversibility of the affine map $f_{A,\bar{a}}$, admits a natural
dynamical interpretation. It may be viewed as a linear cohomological equation (in the sense of classical Livšic theory \cite{L})
over the toral automorphism $\bar{x} \mapsto A\bar{x}$, where the unknown $\bar{r}$ plays the role
of a transfer function and the term $(AR+I)\bar{a}$ represents a cocycle twisted by
the reversing symmetry $R$.

From this perspective, reversibility of $f_{A,\bar{a}}$ is equivalent to the absence of a cohomological obstruction: the translation component $\bar{a}$ must be cohomologous, via $\bar{r}$, to its image under the reversing symmetry. When $1$ is not an eigenvalue of $A$, the operator $A-I$ is surjective on $\mathbb{T}^2$, and the cohomological equation always admits a solution. This explains why in this case reversibility of the affine map reduces to reversibility of its linear part. In contrast, in the hyperbolic case the solvability of the cohomological equation imposes a genuine arithmetic restriction on the translation, which appears in Theorem~\ref{revtoraut} as the lifting condition
$(AR+I) \bar a \in \operatorname{Im}(A-I)$.

Thus, Theorem~\ref{revtoraut} provides a complete classification of
reversibility for affine toral automorphisms: elliptic and parabolic cases
do not have any cohomological obstruction, while hyperbolic dynamics carry a 
nontrivial one.

\section*{Acknowledgements}

The first author was supported partially by the Department of Science and Technology (DST), Govt.~of India, under the Scheme “Fund for Improvement of S\&T Infrastructure” (FIST) [File No.~SR/FST/MS-I/2019/41]. The research of the second author was funded by UGC [NTA Ref.~No.~201610319430], Govt.~of India. The third author acknowledges the ANRF  research grant ANRF/ARGM/2025/000122/MTR. 


\end{document}